\documentclass[11pt]{article}
\usepackage{amsmath, amssymb, theorem, latexsym, epsfig}
\numberwithin{equation}{section}

\theoremstyle{plain}
\theorembodyfont{\itshape}
\newtheorem{theorem}{Theorem}[section]
\newtheorem{proposition}[theorem]{Proposition}
\newtheorem{lemma}[theorem]{Lemma}
\newtheorem{corollary}[theorem]{Corollary}

\theorembodyfont{\rmfamily}
\newtheorem{definition}[theorem]{Definition}

\newtheorem{remark}[theorem]{Remark}

\newenvironment{proof}{{\noindent \textbf{Proof}\,\,}}{\hspace*{\fill}$\Box$\medskip}
\newcommand{\eps}{\varepsilon}
\newcommand{\Sing}{\operatorname{Sing}}

\title{On polynomially integrable planar outer billiards
and curves with symmetry property}
\author{A.Glutsyuk\thanks{ CNRS, France (UMR 5669 (UMPA, ENS de Lyon) and UMI 2615 (Lab. J.-V.Poncelet)), Lyon, France.
E-mail:
aglutsyu@ens-lyon.fr}
\thanks{National Research University Higher School of Economics (HSE), Moscow, Russia}
 \thanks{Supported by part by RFBR grants 13-01-00969-a, 16-01-00748, 16-01-00766
 and ANR grant ANR-13-JS01-0010.},  E.Shustin\thanks{School of Math. Sci., Tel Aviv University,
 Ramat Aviv, Tel Aviv 69978, Israel. E-mail: shustin@post.tau.ac.il}
 \thanks{Supported by the grant 176/15 from the Israeli Science Foundation.}}

\begin{document}
\maketitle
\begin{abstract} We show that every polynomially integrable planar outer convex billiard is elliptic. We also
prove an extension of this statement to non-convex billiards.
\end{abstract}

\tableofcontents

\def\cc{\mathbb C}
\def\oc{\overline{\cc}}
\def\oci{\oc_{\infty}}
\def\cp{\mathbb{CP}}
\def\wt#1{\widetilde#1}
\def\rr{\mathbb R}
\def\var{\varepsilon}
\def\tt{\mathcal T}
\def\mcr{\mathcal R}
\def\a{\alpha}
\def\ha{\hat a}
\def\hb{\hat b}
\def\hc{\hat c}
\def\hd{\hat d}
\def\nn{\mathbb N}
 \def\mct{\mathcal T}
 \def\Int{\operatorname{Int}}
 \section{Introduction, main result and plan of the paper}

\subsection{Introduction and main result}

Let $C\subset\rr^2$ be a smooth closed strictly convex curve. Let $S$ denote the complement of the ambient plane $\rr^2$ to the
closure of the interior of the curve $C$. The {\it (planar) outer billiard} is a dynamical system $\mct:S\to S$ defined as follows. Pick a point
$A\in S$. There are two tangent rays to $C$ issued from the point $A$. Let $R$ denote the right tangent ray: the other tangent ray is
obtained from $R$ by rotation around the point $A$ of angle between zero and $\pi$.  Let $P$ denote the tangency point of the ray $R$
with the curve $C$. By definition, {\it the image $\mct(A)$ is the point of the
ray $R$ that is symmetric to $A$ with respect to the  point $P$.}

Outer billiards were introduced by J.Moser in \cite{moser} as a toy model for planetary motion.
A mechanical interpretation of outer billiard as an impact oscillator was given by Ph. Boyland
\cite{boyland}. For a survey on outer billiards see \cite{tab95, tab05, tabdog}. It is well-known that the outer billiard map preserves the

standard Euclidean area, and hence, it is a symplectomorphism.

The famous Birkhoff Conjecture deals with convex bounded planar billiards having a smooth boundary: the classical (Birkhoff)
billiards with the usual reflection law,
the angle of incidence equals the angle of reflection.  Recall that a {\it caustic} of a planar billiard $\Omega$
  is a curve $C$
 such that each tangent line to $C$ reflects from the boundary of the billiard to a line tangent to $C$.
 A billiard  is called {\it Birkhoff caustic-integrable,} if a neighborhood of its boundary  is foliated by closed caustics.
 The {\bf Birkhoff Conjecture}  states that {\it the only integrable billiard is an ellipse.}

The Birkhoff Conjecture  motivated an analogous conjecture for the outer billiards.
Namely, let $C\subset\rr^2$ be a smooth closed strictly convex curve, $S\subset\rr^2$ be the
exterior component of its complement. We say that the corresponding outer billiard map
$\mct:S\to S$ is
{\it integrable},  if some neighborhood $U$ of the curve $C$ in $\overline S$ admits a smooth $\mct$-invariant function $U\to\rr$ (called a {\it first integral}) that has no critical points on $C$.
It is known that every elliptic outer billiard is integrable. S.L.Tabachnikov's conjecture
\cite[p. 101]{tab08} states that the converse is also true:
 if an outer billiard is integrable, then it is an ellipse. We prove this conjecture for polynomial first integrals.

\begin{definition} A planar outer billiard is {\it polynomially integrable,} if there exists a non-constant polynomial $f(x,y)$ that is  invariant under the outer billiard mapping.
\end{definition}

\begin{remark} Polynomiality of the integral is a very strong restriction of the general Tabachnikov's Conjecture.
At the same time,  the condition of
just non-constancy of a polynomial integral does not forbid it to have critical points on the billiard boundary, while
the definition  of general integrability forbids them.
 \end{remark}

The main result of the paper is the following theorem.

\begin{theorem} \label{tab} Let a planar outer billiard generated by a $C^4$-smooth strictly convex closed
curve $C$ be polynomially integrable. Then $C$ is an ellipse.
\end{theorem}

A particular case of Theorem \ref{tab} under some non-degeneracy assumptions
 was proved by S.L.Tabachnikov  \cite[theorem 1]{tab08}. These assumptions imply in particular
 that the complex projective Zariski closure of the curve $C$ is  non-singular.
See the end of Subsection 1.2 for more details.

We also prove the following more general theorem on non-convex curves generating
multivalued outer billiard mappings, see the next definition.

\begin{definition} A $C^k$-smoothly immersed curve
$C\subset\rr^2$ {\it generates a polynomially integrable multivalued outer billiard,} if
there exists a non-constant polynomial  $f(x,y)$ such that for every $P\in C$ and every $A,B\in T_PC$ symmetric with respect to the point $P$ one has
$f(A)=f(B)$. The latter polynomial is called an {\it integral} of the multivalued outer billiard constructed on the curve $C$.
\end{definition}

\begin{theorem} \label{tabmult} Let $C\subset\rr^2$ be a $C^4$-smoothly immersed image of either an interval,  or a circle, that does not lie in a line.
Let $C$ generate a polynomially integrable multivalued outer billiard.
Then the curve $C$ lies in  a conic.
\end{theorem}

Theorem \ref{tab} follows from Theorem \ref{tabmult}.

For the proof of Theorem \ref{tabmult} we consider the complexification $\cc^2$ of the real plane and the ambient projective plane $\cp^2\supset\cc^2$. Let $\oci$ denote the infinity line:
 $$\oci=\cp^2\setminus\cc^2.$$
It is well-known that the polynomial integral $f$ is constant on $C$ (cf. \cite{tab08}, see Proposition \ref{f=c}
below). This implies that the projective complex Zariski closure $\gamma$
of the curve $C$ is an algebraic curve lying in a level curve of the complex polynomial $f$; say, in the zero
level curve $\Gamma=\{ f=0\}$, provided that $f|_C\equiv 0$. It suffices to show that each  nonlinear
irreducible component $\alpha$ of the curve $\gamma$ is a conic.
S.Tabachnikov's arguments \cite{tab08} together with their generalization
due to M.Bialy  and A.E.Mironov \cite[section 6]{bm} imply that all the singular and inflection points (if any) of the curve $\gamma$ lying in $\alpha$ are contained in the infinity line $\oci$
(Theorem \ref{relsym} in Section 2). The curve $\gamma$ has the
so-called relative symmetry property: for every point $t\in\gamma\cap\cc^2$
the central symmetry of the affine line $T_t\gamma$ with respect to the point $t$ permutes the points of its
intersection with a bigger complex algebraic curve $\Gamma\supset\gamma$
(follows from definition and analyticity).
We study the local branches of the curve
$\alpha$ at points of the intersection $\gamma\cap\oci$. Each local branch is a nonlinear irreducible germ
of analytic curve, thus it admits a local holomorphic parametrization by small complex parameter $t$,
$$t\mapsto(t^q,ct^p(1+o(1))), \ p,q\in\mathbb N, \ p>q\geq1, \ c\neq0$$
in local affine chart centered at the base point.
We show that the relative symmetry property (and even its weaker, local version)
implies that each local branch
transverse to $\oci$ is {\it quadratic:}   $p=2q$
(Theorem \ref{tsym} and its corollary).
This is done in Section 3 via asymptotic analysis of  the relative symmetry property. Finally,
$\alpha$ has the two following properties: all its singular and inflection points
(if any) lie in the infinity line $\oci$; each its local branch transverse to $\oci$ is quadratic. This together with the
next new algebro-geometric  theorem proved in  Section 4 will imply Theorem \ref{tabmult}.

\begin{theorem} \label{shust} Let a nonlinear irreducible algebraic curve $\alpha\subset\cp^2$ have neither singular, nor inflection points in an affine chart
$\cc^2\subset\cp^2$. Let each of
its local branches at every point in $\alpha\cap\oci$ that is transverse to $\oci$ (if any)
be subquadratic:  $p\leq 2q$ in its above parametrization.
Then $\alpha$ is a conic.
\end{theorem}

The proof of Theorem   \ref{tabmult} will be given in Subsection 4.3.

\begin{remark} It is well-known that polynomial integrability of a Birkhoff billiard $\Omega$ is equivalent to the existence of a non-trivial
first integral of the billiard flow in a neighborhood in $T\rr^2|_{\overline{\Omega}}$ of the zero section of the bundle $T\rr^2|_{\partial\Omega}$
 that is analytic in the speed. In more detail here analyticity is required to be uniform
 so that the convergence radius of its Taylor series in the speed be  uniformly
 bounded from below. A direct reformulation of this assertion for the outer billiards would state that the existence
of an analytic first integral of an outer billiard in a neighborhood of its boundary implies the existence of a polynomial integral. This is unknown, but  this would reduce  the Tabachnikov's Conjecture for analytically integrable outer
billiards to the polynomial case and hence, together with
the present results, would prove it. To our opinion,  the analytic Tabachnikov's Conjecture
seems to be of a level comparable to that of the general smooth case.
\end{remark}

\subsection{Historical remarks}

 The classical Birkhoff Conjecture was studied by many mathematicians starting from the famous paper by H.Poritsky  \cite{poritsky}, where he  stated it in print  and
 proved it under the additional assumption that
 the billiard in each closed caustic near the boundary has the same closed caustics, as the initial billiard
 (see also  \cite{amiran}). Here we will mention just few of most known results.
In 1993 M.Bialy \cite{bialy} proved that if the phase cylinder of the billiard is
 foliated (almost everywhere) by continuous curves which are invariant under the billiard map,  then the boundary curve is a circle (see also
\cite{wojt}). See also Bialy's papers
\cite{bialy1, bialy2} for similar results on billiards in constant curvature and on magnetic billiards respectively.
D.V.Treschev's experiences  \cite{treshchev, tres2, tres3} provide a numerical
evidence of existence of the so-called locally integrable billiards,  where a germ of the second iterate of the
billiard map is conjugated to a rigid rotation (in two and higher dimensions).
 Recently V.Kaloshin and A.Sorrentino have proved a {\it local version} of the Birkhoff Conjecture \cite{kalsor}:
 {\it an integrable deformation of
 an ellipse is an ellipse} (see \cite{kavila} for the case of ellipses with small eccentricities). Birkhoff Conjecture motivated the so-called
{\it Algebraic Birkhoff Conjecture}, which deals with the {\it polynomially integrable planar billiards,}
where the billiard geodesic flow has a first integral that depends polynomially
on the speed components and is non-constant on the unit level hypersurface of the norm of the speed. It states that the only polynomially
integrable convex billiard with smooth boundary is an ellipse. The study of the Algebraic Birkhoff Conjecture and its appropriate generalization to (not necessarily convex) billiards with piecewise smooth boundary on surfaces of constant curvature was started  by S.V.Bolotin in
 \cite{bolotin, bolotin2}, see also a survey in \cite[chapter 5, section 3]{kozlov}. It was proved for polynomial integrals of degrees
 up to 4 in \cite{bm2}. Its complete proof for billiards on  any surface of constant curvature
 was recently obtained as a result of the three following papers:
 two joint papers by M.Bialy and A.E.Mironov \cite{bm, bm3}; a very recent paper \cite{gl}
 (see also its short version \cite{gl2}) of the first author of  the present article.
  For more detailed surveys on Birkhoff Conjecture and its algebraic version see \cite{bm, gl, kalsor} and references therein.

For outer billiards a particular case of Theorem \ref{tab} was proved by S.L.Tabachnikov  \cite[theorem 1]{tab08} under the following additional assumptions:

(i) the complex
Zariski closure of the curve  $C$ in $\cp^2$ is a non-singular algebraic curve;

(ii) the gradient $\nabla F$ of the polynomial integral $F$ does not vanish identically on $C$.

Tabachnikov
introduced a powerful
method that allowed him to show that under assumptions (i) and (ii)
all the singular and inflection points (if any) of the complex projective
Zariski closure of the curve $C$ lie in the infinity line. M.Bialy and A.E.Mironov introduced  a modified version of
Tabachnikov's method that allowed them to prove the same result in the general case
(their private communication, see Theorem \ref{relsym} below and its proof) and to prove a similar statement in the context of the Algebraic Birkhoff Conjecture in constant curvature \cite{bm, bm3}.

\begin{remark} The second part of the proof of the Algebraic Birkhoff Conjecture given in \cite{gl, gl2} uses results
of  \cite{bm, bm3} and techniques elaborated in the present article and in  the previous paper of the first author \cite{alg}.
\end{remark}

\section{Complexification of the curve $C$: singularities and inflection points}

The following proposition is a version of a  result from \cite{tab08}.

  \begin{proposition} \label{f=c}  Let $C\subset\rr^2$ be a $C^1$-smoothly immersed
  image of  either an interval, or a circle. Let $C$
generate a polynomially integrable multivalued outer billiard with the integral
$f(x,y)$.  Then one has $f|_C\equiv const$.
\end{proposition}

\begin{proof} For every point $P\in C$ the restriction of the polynomial integral $f$ to the affine tangent line
$T_PC$ is invariant under the central symmetry with respect to the point $P$, by definition.
Hense, its derivative at $P$ vanishes, as does the derivative at 0 of an even function. Thus, the polynomial
$f$ has zero derivative along a vector tangent to $C$ at $P$. Therefore, $f|_C\equiv const$, by connectivity. This proves the proposition.
\end{proof}

\begin{theorem} \label{relsym}   Let $C\subset\rr^2$ be a $C^1$-smoothly immersed
  image of  either an interval, or a circle, that does not lie in a line and
generates a polynomially integrable multivalued outer billiard.
Its complex projective Zariski closure  $\gamma\subset\cp^2$ is an algebraic curve.
Each nonlinear irreducible component  of the curve $\gamma$
 contains neither singular, nor inflection points of the affine curve $\gamma\cap\cc^2$
\end{theorem}

A particular case of Theorem \ref{relsym} under assumptions (i) and (ii) from Subsection 1.2 was proved  by
S.L.Tabachnikov in \cite{tab08} (for convex curves $C$, but his argument works without convexity
assumption). M. Bialy and A. E. Mironov
have extended his proof to  the general case by  using their ideas from  \cite[section 6]{bm}.
This proof of Theorem \ref{relsym} due to Tabachnikov, Bialy, and Mironov is given below.

\begin{proof} {\bf of Theorem \ref{relsym}.}
Let $f$ be a polynomial integral normalized so that $f|_{C}\equiv0$. Then $f|_{\gamma}\equiv0$ by definition, and
this implies that $\gamma$ is an algebraic curve.
Let $\alpha\subset\gamma$ be a nonlinear
irreducible component of the curve $\gamma$.
Let $\Psi$ be an irreducible polynomial vanishing on $\alpha$. Then
$$f=g(x,y)\Psi^m(x,y), \ m\in\nn, \ g|_{\alpha}\not\equiv0.$$
Set\footnote{The case, when $m=1$, was treated by S.L.Tabachnikov \cite{tab08}. His arguments
were extended to arbitrary $m$ by M.Bialy and A.E.Mironov, who introduced the function $F$ (their private communication, which repeats their arguments from \cite[section 6]{bm})}
$$F(x,y)=f^{\frac1m}(x,y)=g^{\frac1m}(x,y)\Psi(x,y).$$
The function $F$ is multivalued algebraic, and any two of its holomorphic leaves
differ by multiplication by $m$-th root of unity. Its branching locus
is contained in the curve $\{ g=0\}$.

\def\Sing{\operatorname{Sing}}

For every $P\in\alpha\cap\cc^2$ and for every two points $A,B\in T_P\alpha$ symmetric with respect to the point $P$ one has $f(A)=f(B)$: this equality holds in the real domain with $P\in C$ (by definition) and extends analytically to the complex domain.
This implies that for every point $P\in\alpha$ such that $g(P)\neq0$ and $P$ is not a singular point of the curve
$\alpha$ each leaf of the function $F$ holomorphic on a neighborhood of the point $P$
 has local symmetry property: for every  $A,B\in T_P\alpha$ symmetric with respect to the point $P$ and
 close enough to it one has $F(A)=F(B)$.

Consider the (multivalued)
vector field $V=F_y\frac{\partial}{\partial x}-F_x\frac{\partial}{\partial y}$, which is tangent to the level curves of the function $F$ and
does not vanish identically on $\alpha$.
The above local symmetry property is equivalent to the statement that the function
\begin{equation} U(x,y,\var)=F(x+\var F_y,y-\var F_x)\label{fxy}\end{equation}
is even in $\var$ for all $P=(x,y)\in\alpha\setminus(\Sing(\alpha)\cup\{ g=0\})$.
Equivalently,  its Taylor series in $\var$ should contain only even powers of the variable $\var$. Set
$$H(F)=F_{xx}F_y^2-2F_{xy}F_xF_y+F_{yy}F_x^2.$$
One has
\begin{equation} H(F)|_{\alpha}\equiv const\label{hconst}\end{equation}
locally for each leaf of the function $F$ over $\alpha\setminus(\Sing(\alpha)\cup\{ g=0\})$.
Indeed, for every $P=(x,y)\in\alpha\setminus(\Sing(\alpha)\cup\{ g=0\})$  the Taylor coefficient at $\var^3$ of the function $U(x,y,\var)$, which should vanish,
 equals $\frac{dH(F)}{dV}$ up to constant factor, see
 \cite[lemma 2]{tab08} and \cite[section 6]{bm}. Thus, the latter derivative vanishes, and hence, $H(F)|_{\alpha}\equiv const$,
 since $V$ is tangent to $\alpha$.

The restriction $H(\Psi)|_{\alpha}$ coincides with the value of the Hessian quadratic form
of the function $\Psi$ on its skew gradient. Recall that $H(F)=g^{\frac3m}H(\Psi)$ on the curve
$\alpha=\{\Psi=0\}$ (see also \cite[lemma 2]{tab08}).
Therefore, $H(F)$ vanishes over all the
singular and inflection points of the affine curve $\alpha\cap\cc^2$. For every regular point $A\in\alpha\cap\cc^2$
that is singular for the curve $\gamma$ one has $g(A)=0$, hence, $H(F)(P)\to0$, as $P\to A$ along the curve
$\alpha$. This together with (\ref{hconst}) implies that
$\alpha\cap\cc^2$ should be a straight line, as soon as it contains either a singular, or an inflection point of the
curve $\gamma$.
Thus, since
$\alpha\cap\cc^2$ is not a line by assumption, it contains neither singular, nor inflection points of the curve $\gamma$.
 Theorem \ref{relsym} is proved.
\end{proof}

\section{Relative symmetry property and quadraticity}

\begin{definition} Recall that a {\it local branch} of an analytic curve $\gamma$ at a point $A\in\gamma$
is an irreducible component of its germ at $A$.
\end{definition}

In the present section we study  nonlinear
local branches of the complex projective Zariski closure
$\gamma$ of the curve $C$ at points of the intersection $\gamma\cap\oci$ and show that each branch  transverse to
$\oci$ is quadratic (Theorem \ref{tsym} and its corollary, see Subsection 3.1). Theorem \ref{tsym}
is stated in a more general context, for an irreducible germ $b$ of analytic curve at a point $A\in\oci$ satisfying the
so-called local relative symmetry property introduced in Subsection 3.1.
In the proof of Theorem \ref{tsym} given in Subsection 3.3
we use preparatory results  given partly in \cite{alg}   on the
asymptotics of the intersection points of the tangent line $T_Pb$ with a
given irreducible germ $a$ of analytic curve at $A$, as $P\to A$ (Propositions \ref{pasym}, \ref{asym}
and their corollaries in Subsection 3.2).

\subsection{Relative symmetry property}
 \begin{definition}  \label{def1} Let $\cc^2\subset\cp^2$ be a fixed affine chart. Let $\gamma\subset\cp^2$
 be an irreducible algebraic curve distinct from a line. We say that $\gamma$ has
 {\it relative symmetry property}, if there exists an algebraic curve
$\Gamma\subset\cp^2$ containing $\gamma$
such that for every $t\in\gamma\cap\cc^2$
 the intersection $T_t\gamma\cap\Gamma\cap\cc^2$ is symmetric with respect to the point $t$ as a subset of the affine
complex  line $T_t\gamma\cap\cc^2$: it is invariant under the central symmetry $x\mapsto -x$ in an affine coordinate $x$ on $T_t\gamma$
centered at $t$.
\end{definition}

We will also deal with the local version of the relative symmetry property. To state it, let us introduce the following definition.
\begin{definition} Let $L\subset\cp^2$ be a line, and let $A\in L$.
A {\it $(L,A)$-local multigerm} is a finite union
of distinct irreducible germs of analytic curves $b_1,\dots,b_N$ (called {\it components}) at base points  $A_j\in L$
such that each germ at  $A_j\neq A$ is different from the line $L$. (A germ at $A$ can be arbitrary,
in particular, it may coincide with the germ $(L,A)$.)
The {\it $(L,A)$-localization} of an algebraic curve
in $\cp^2$ is the corresponding $(L,A)$-local
multigerm formed by all its local branches of the above type.
\end{definition}

\begin{definition} Let  $\cc^2\subset\cp^2$ be a given affine chart. Let
$A\in\cp^2$, $b\subset\cp^2$ be a nonlinear irreducible germ of analytic curve at $A$.
The germ $b$ has  {\it (local) relative symmetry property,} if there exists a $(T_Ab,A)$-local
multigerm  $\Gamma$
containing $b$  such that for every $t\in b\cap\cc^2$ close enough to $A$ the intersection $T_tb\cap\Gamma\cap\cc^2$
is symmetric with respect to the point $t$.
\end{definition}

Consider an irreducible  nonlinear germ $b$ of  analytic curve in $\cp^2$ at a given point $A$. Let us choose
affine coordinates $(z,w)$ centered at $A$ so that the tangent line $T_Ab$ be the $z$-axis. Then one can find
a local
bijective parametrization of the germ $b$ by a complex parameter $t\in(\cc,0)$ of the type
\begin{equation}
t\mapsto(t^{q},c_bt^{p}(1+o(1))), \quad
q=q_b,\ p=p_b\in\nn, \ 1\leq q<p, \ c_b\neq0;\label{curve} \end{equation}
$$q=1 \text{ if and only if } b \text{ is a smooth
germ.}$$
In the case, when $b$ is a germ of line, it is parameterized by $t\mapsto(t,0)$, and we set $q_b=1$,
$p_b=\infty$.
\begin{definition}
The {\it projective
Puiseux exponent} \cite[p. 250, definition 2.9]{alg} of a nonlinear germ $b$ is the ratio
$$r=r_b=\frac{p_b}{q_b}.$$
The germ $b$ is called {\it quadratic}, if $r_b=2$, and is called
{\it subquadratic}, if $r_b\le2$. By definition, the projective Puiseux exponent of a germ of line equals infinity.
\end{definition}

\begin{theorem} \label{tsym} Let a nonlinear irreducible germ $b$ of
analytic curve in $\cp^2$ at a point  $A\in\oci$ be transverse to
$\oci$ and have  local relative symmetry property. Then it is quadratic.
\end{theorem}

Theorem \ref{tsym} is proved in Subsection 3.3.

\begin{corollary} \label{cortsym} Let a nonlinear irreducible germ $b$ of
 analytic curve in $\cp^2$ at a point  $A\in\oci$ be transverse to $\oci$ and lie in the complex projective
Zariski closure $\gamma$ of a $C^2$-smoothly immersed real  curve  generating a
 polynomially integrable multivalued outer billiard. Then $b$ is quadratic.
 \end{corollary}

\begin{proof} Let $C$ be the real curve under question, and let $f$ be a polynomial integral of the
corresponding outer billiard. Without loss of generality we consider that $C$ is an immersed image of a connected curve,
either an interval, or a circle (passing to a smaller curve) and
$f|_{\gamma}\equiv0$, see Proposition \ref{f=c}. Let $\gamma_0$ denote the
 irreducible component  of the curve $\gamma$ containing the germ $b$. The component $\gamma_0$
 is nonlinear and has the relative symmetry property with respect to the algebraic curve
 $\Gamma=\{ f=0\}\supset\gamma$: for every $P\in\gamma\cap\cc^2$   the restriction  of the integral $f$
to the affine tangent line $T_P\gamma$  is invariant under the central symmetry with respect to $P$,
as in the proof of Theorem \ref{relsym}.

For every point $u\in b$ close to $A$ the
intersection of the line $T_ub$ with the algebraic curve
$\Gamma$ coincides with the intersection of the line $T_ub$  and the
$(T_Ab,A)$-localization of the curve $\Gamma$.
This is implied by definition and the following facts:

- for any neighborhood $U$ of a given line $L$ all the lines sufficiently close to $L$ are
contained in $U$;

- the point of intersection $T_ub\cap T_Ab$ tends to $A$, as
$u\to A$.

The local branch $b$ of the curve $\gamma_0$
 has local relative symmetry property with respect to the
$(T_Ab,A)$-localization of the curve $\Gamma$, by the above discussion. This together with
Theorem \ref{tsym} implies that $b$ is quadratic and proves the corollary.
\end{proof}

\def\ga{\gamma}
\def\hga{\hat\gamma}

\subsection{Asymptotics of intersection points}

\begin{proposition} \label{pasym}  Let $a$, $b$ be transverse irreducible germs of holomorphic curves at the origin in $\cc^2$, let $b$ be nonlinear.
Let $(z,w)$ be affine
coordinates in a neighborhood of the origin in $\cc^2$ centered at $0$ such that $b$ is tangent to the $z$-axis at $0$. Let\; $t$ be
the local parameter of the curve $b$  as in (\ref{curve}): $z(t)=t^{q_b}$;  $w(t)=c_bt^{p_b}(1+o(1))$.
Then for every $t$ small enough the intersection $T_tb\cap a$  consists
of $q_a$ points $\xi_1,\dots,\xi_{q_a}$  whose coordinates have the following asymptotics, as $t\to0$:
\begin{equation} z(\xi_j)=O(t^{p_b})=o(t^{q_b})=o(z(t)), \ w(\xi_j)=(1-r_b)w(t)(1+o(1)).\label{zwb}\end{equation}
(Recall that $q_a=1$, if $a$ is a germ of line.)
\end{proposition}

\begin{proof} It suffices to prove just the second asymptotic formula in (\ref{zwb}). Indeed, one has $z(\xi_j)=O(w(\xi_j))$,
by transversality. This together with the second  formula
in (\ref{zwb}) implies the first one: $z(\xi_j)=O(w(t))=O(t^{p_b})$.

For every $t$ small
enough the tangent line $T_tb$ intersects the $z$-axis at a point $P_t$ with the coordinate
\begin{equation}z(P_t)=\nu z(t)(1+o(1))=\nu t^{q_b}(1+o(1)), \ \nu=\frac{r_b-1}{r_b},\label{zpt}\end{equation}
by \cite[Proposition 2.10, p. 250]{alg}. Let $Q_t$ denote the intersection point of the line $T_tb$ with the $w$-axis. One has
\begin{equation}w(Q_t)=\frac{\nu}{\nu-1}w(t)(1+o(1))=(1-r_b)w(t)(1+o(1)).\label{trisim}\end{equation}
Indeed,
the triangle with  the vertices $P_t$, $t$, $(z(t),0)$ is ``complex-similar'' to the triangle
$P_tQ_tO$ ($O$ is the origin), since their sides opposite to the vertex $P_t$ lie in parallel affine complex lines. That is, in the new affine coordinates
centered at $P_t$
the second triangle is obtained from the
first one by multiplication by the complex number $\frac{z(P_t)}{z(P_t)-z(t)}=\frac{\nu}{\nu-1}(1+o(1))$, see (\ref{zpt}). This
implies (\ref{trisim}).
Let now $a$ be an arbitrary irreducible germ of holomorphic curve at the origin that is transverse to $b$.
Every family of points $\zeta(t)$ of the intersection  $T_tb\cap a$ has $w$-coordinate asymptotically equivalent to $w(Q_t)$, by transversality and
since the line $T_tb=Q_t\zeta(t)$ tends to the $z$-axis. In more detail, suppose the contrary:
the above $w$-coordinates are not asymptotically equivalent, hence the difference $w(\zeta(t))-w(Q_t)$ is asymptotically no less than $cw(\zeta(t))$, $c\neq0$. Then
the ratio $\frac{w(\zeta(t))-w(Q_t)}{z(\zeta(t))}$ would not tend to zero, since $z(\zeta(t))=O(w(\zeta(t)))$,
as was mentioned above. Hence, the line $Q_t\zeta(t)$ would not tend to the $z$-axis, -- a contradiction.
This together with (\ref{trisim}) proves the second equality in (\ref{zwb}).
The proposition is proved.
\end{proof}

\begin{corollary} \label{cortr}
Let $A\in\oci$, and let $(b,A)\subset\cp^2$ be a nonlinear irreducible germ of
analytic curve that is transverse to $\oci$.
Let $(a,A)$ be another irreducible germ transverse to $b$, and let $\xi_1,\dots,\xi_{q_a}$, $\xi_j=\xi_j(t)$ be the points of intersection
$T_tb\cap a$. Let $(x,y)$ be affine coordinates in $\cc^2$ such that the $x$-axis is tangent  to $b$ at $A$. Then
\begin{equation}x(t)=o(x(\xi_j(t)))\quad \text{for all } j,\quad \text{as } t\to A.\label{xio}\end{equation}
\end{corollary}

\begin{proof} Take the local coordinates $(z,w)=(\frac1x, \frac yx)$ centered at $A$ and apply the first formula in (\ref{zwb}).
\end{proof}

\begin{proposition}  \label{asym} (cf. \cite[p. 268, Proposition 2.50]{alg}\footnote{The formulas from loc. cit. provide the inverse expressions,
for the coordinate $z(t)$ in terms of $z(\xi_j(t))$. They are equivalent to the formulas given here.}) Let $a$, $b$ be irreducible germs of holomorphic
curves at the origin in the plane $\cc^2$ with coordinates $(z,w)$. Let $a$ and $b$ be  tangent to the $z$-axis, and let
$b$ be nonlinear.
 Let $c_a$ and $c_b$ be the corresponding coefficients in (\ref{curve}).
Then for every $t$ small enough the intersection \mbox{$T_tb\cap a$} consists
of $p_a$ points $\xi_1,\dots,\xi_{p_a}$ (or just one point $\xi_1$, if $a$ is the germ of the line $T_Ob$) whose $z$-coordinates have the following asymptotics, as $t\to0$.

Case 1): $r_a>r_b$ (including the linear case, when $r_a=\infty$).  One has
\begin{equation}z(\xi_j)=\frac{r_b-1}{r_b}z(t)(1+o(1))=\frac{r_b-1}{r_b}t^{q_b}(1+o(1)) \text{ for } 1\leq j\leq q_a,\label{a>b1}\end{equation}
\begin{equation}z(t)=O((z(\xi_j))^{\frac{r_a-1}{r_b-1}})=o(z(\xi_j)) \text{ for } j> q_a.\label{a>b2}\end{equation}
(Points satisfying (\ref{a>b2}) exist only when $a$ is nonlinear.)

Case 2): $r_a=r_b$. One has
\begin{equation} z(\xi_j)=\zeta_j^{q_a}z(t)(1+o(1))= \zeta_j^{q_a}t^{q_b}(1+o(1)),\label{a=b}\end{equation}
where $\zeta_j$ are the roots of the polynomial
\begin{equation}R_{p_a,q_a,c}(\zeta)=c\zeta^{p_a}-r\zeta^{q_a}+r-1; \ r=\frac{p_a}{q_a}>1, \ c=\frac{c_a}{c_b}.
\label{rpq}\end{equation}
(In the case, when $b=a$, one has $c=1$, and the above polynomial has the double root 1 corresponding to the
tangency  point $t$.)

Case 3): $r_a<r_b$. One has
\begin{equation} z(\xi_j)=O((z(t))^{\frac{r_b}{r_a}})=o(z(t)). \nonumber
\end{equation}
\end{proposition}

\begin{proof} All the statements of the proposition were proved in loc. cit. except for the statement saying that in Case 1) one has
exactly $q_a$ intersection points with asymptotics (\ref{a>b1}) and exactly $p_a-q_a$ intersection points with asymptotics
(\ref{a>b2}) (no points satisfying (\ref{a>b2}), if $a$ is linear). Let us prove the latter statement. Thus, we consider that $r_a>r_b$. The case, when $a$ is linear, was treated in \cite[proposition 2.10, p.250]{alg},
see also formula (\ref{zpt}) above.
Thus, we consider that $a$ is nonlinear. Let $\tau$ denote the value of the local parameter
of the curve $a$ at a point of intersection $\xi(t)\in T_tb\cap a$. The obvious analytic equality
\begin{equation} G(t,\tau)=w(t)+\frac{w'(t)}{z'(t)}(z(\xi(t))-z(t))-w(\xi(t))=0\label{eqob}\end{equation}
has asymptotic form
\begin{equation} G(t,\tau)=t^{p_b}(1+o(1))+r_bt^{p_b-q_b}(\tau^{q_a}-t^{q_b})(1+o(1))-c\tau^{p_a}(1+o(1))=0, \label{aseq}\end{equation}
$ c=\frac{c_a}{c_b}\neq0$,
as in \cite[p. 269, proof of Proposition 2.50]{alg}. The Newton diagram of the germ of analytic function $G(t,\tau)$
is generated by the three monomials:
$(1-r_b)t^{p_b}$, $r_bt^{p_b-q_b}\tau^{q_a}$, $-c\tau^{p_a}$. It consists of two edges: the first one
with the vertices $(p_b,0)$ and $(p_b-q_b,q_a)$; the second one with the vertices $(p_b-q_b,q_a)$ and $(0,p_a)$. The latter edges
lie on distinct lines. The three above statements on Newton diagram
follow from the inequality $r_a>r_b$, as in loc. cit. The germ of analytic curve
$\{ G=0\}\subset\cc^2_{(t,\tau)}$ at the origin is a union of two germs $\eta_1\cup\eta_2$ corresponding to the edges of the Newton diagram, as in loc. cit.
Namely, the monomials $(1-r_b)t^{p_b}$ and $r_bt^{p_b-q_b}\tau^{q_a}$ generating the first edge are asymptotically opposite
(asymptotically ``cancel out'') along the germ $\eta_1$, and all the other Taylor monomials of the function $G$ are of higher order
along $\eta_1$, as in loc. cit. This implies that for every fixed small $t$ there are exactly $q_a$ parameter values $\tau$ for which
 $(t,\tau)\in\eta_1$, and they satisfy asymptotic equality (\ref{a>b1}).  Similarly, the monomials
 $r_bt^{p_b-q_b}\tau^{q_a}$ and $-c\tau^{p_a}$ are asymptotically opposite along the germ $\eta_2$, and for every fixed small $t$
 there are exactly $p_a-q_a$ values $\tau$ such that $(t,\tau)\in\eta_2$, and they satisfy asymptotic equality
 (\ref{a>b2}).
  Proposition \ref{asym} is proved.
\end{proof}

\begin{corollary} \label{casym}
Let $a,b\subset\cp^2$ be irreducible germs of holomorphic curves at a point $A\in\oci$ that are tangent to each other and
transverse to $\oci$, let $b$ be nonlinear.
Let $(x,y)$ be affine coordinates on $\cc^2$ with the $x$-axis being tangent to $b$ at $A$. Let
$\xi_1,\dots,\xi_{p_a}$, $\xi_j=\xi_j(t)$ be the points of intersection $T_tb\cap a$. Their $x$-coordinates
have the following asymptotics, as $t\to A$:

Case 1): $r_a>r_b$.  One has
\begin{equation}x(\xi_j)=\frac{r_b}{r_b-1}x(t)(1+o(1)) \text{ for } 1\leq j\leq q_a,\label{a>b1*}\end{equation}
\begin{equation}x(\xi_j)=o(x(t)) \text{ for } j>q_a.\label{a>b2*}\end{equation}
(Points satisfying (\ref{a>b2*}) exist only when $a$ is nonlinear.)

Case 2): $r_a=r_b$. One has
\begin{equation} x(\xi_j)=\theta_j^{q_a}x(t)(1+o(1)),\label{a=b*}\end{equation}
where $\theta_j$ are the roots of the polynomial\footnote{In the case, when $b=a$, one has $c=1$, and the  polynomial $Q_{p_a,q_a,1}$
has double root 1 corresponding to the tangency  point $t$. It has  roots $\theta$ with $\theta^{q_a}\neq 1$,
if and only if $r=\frac{p_{a}}{q_{a}}\neq2$.}
\begin{equation}Q_{p_a,q_a,c}(\theta)=\theta^{p_a}R_{p_a,q_a,c}(\theta^{-1})=(r-1)\theta^{p_a}
-r\theta^{p_a-q_a}+c; \label{rpq*}\end{equation}
$$r=\frac{p_a}{q_a}=\frac{p_b}{q_b}>1, \ c=\frac{c_a}{c_b}.$$

Case 3): $r_a<r_b$. One has
\begin{equation} x(t)=o(x(\xi_j)).\nonumber
\end{equation}

\end{corollary}

The corollary follows from Proposition \ref{asym} by writing its asymptotics in the local coordinates
$(z,w)=(\frac1x,\frac yx)$.

\subsection{Intersections with the germs having the same
projective Puiseux exponents. Proof of Theorem \ref{tsym}}

Let $A\in\oci$, and let $b\subset\cp^2$ be a nonlinear irreducible germ of analytic curve at
$A$ that is transverse to $\oci$. Let
$\Gamma\supset b$ be an arbitrary $(T_Ab,A)$-local multigerm containing $b$. Let $(x,y)$ be affine
coordinates on $\cc^2$ such that the $x$-axis is tangent to $b$ at $A$.

\begin{proposition} \label{asynt}
Those points of the intersection $T_tb\cap\Gamma$ whose $x$-coordinates are
asymptotically equivalent to $x(t)$ up to a
nonzero multiplicative constant, as $t\to A$, lie in the intersection of the line $T_tb$ with those irreducible germs $a\subset\Gamma$ that are centered at $A$,
tangent to $b$ and for which $r_a\geq r_b$, including the germ $b$. (Some of these germs $a$ may be linear.)
 The asymptotics of their $x$-coordinates are given by either (\ref{a>b1*}) if $r_a>r_b$,
 or (\ref{a=b*}) if $r_a=r_b$.
\end{proposition}

The proposition follows immediately from Corollaries \ref{cortr} and \ref{casym} and the fact that the $x$-coordinates of the points of the intersection of the line $T_tb$ with those germs $a$ in $\Gamma$ that are not centered at $A$ tend to
finite limits, as $t\to A$, while $x(t)\to\infty$.

In the sequel we
assume that the curve $b$ has relative symmetry property with respect to the multigerm
$\Gamma$. First we prove Theorem \ref{tsym} in the next simplest special case in order to underline the main ideas,
and then in the general case.

{\bf Special case:}
$\Gamma\setminus b$ is a union of germs centered at $A$
and transverse to $b$. Let us prove quadraticity of the germ $b$.
To do this, we consider those intersection points of the tangent line $T_tb$ with $\Gamma$, whose $x$-coordinates have
asymptotics $\nu x(t)(1+o(1))$, $\nu\neq0$, as $t\to A$. These are exactly the points of the intersection $T_tb\cap b$ (Proposition \ref{asynt}).
Their $x$-coordinates have asymptotics $\theta_j^qx(t)(1+o(1))$, where $\theta_1,\dots,\theta_p$ are the roots of the polynomial
$W(\theta)=Q_{p,q,1}(\theta)=(r-1)\theta^p-r\theta^{p-q}+1$; here $p=p_b$ and $q=q_b$ are  the degrees from
the parametrization (\ref{curve}) of the germ $b$ in the local chart $(z,w)=(\frac1x,\frac yx)$
(Corollary \ref{casym}). The intersection points of the line $T_tb$ with the other,  transverse germs of the
curve $\Gamma$ have $x$-coordinates with bigger asymptotics, by
Corollary \ref{cortr}. This together with the relative symmetry property implies that the collection of $x$-coordinates
of the points of intersection $T_tb\cap b$ is invariant
under the symmetry with respect to the point $x(t)$. This implies that the collection of powers $\theta_j^q$ is invariant under the symmetry
with respect to 1. Therefore,
\begin{equation}\sum_{j=1}^p\theta_j^q=p,\label{sumtp}\end{equation}
by symmetry. On the other hand,
\begin{equation}\sum_{j=1}^p\theta_j^q=\frac p{r-1}.\label{sumtq}\end{equation}
Indeed, the latter sum is independent of the free term of the polynomial
$W$, being expressed via elementary symmetric polynomials of degrees at most $q<p$, which are independent of the free term. Therefore,
it equals the sum of the $q$-th powers of nonzero roots of the polynomial $(r-1)\theta^p-r\theta^{p-q}$. All the latter $q$-th powers
are equal to $\frac r{r-1}$, hence their sum equals $\frac{qr}{r-1}=\frac p{r-1}$. This proves (\ref{sumtq}). Formulas
(\ref{sumtp}) and (\ref{sumtq}) together imply that $p=\frac p{r-1}$. Hence, $r=2$.

\medskip

{\bf General case.}
Let $a_1,\dots,a_l\subset\Gamma$ be those irreducible germs that are centered at $A$, tangent to $b$ and have the same projective Puiseux exponent:
$r_{a_i}=r_b=r$; $b$ is one of them. Here $l$ may be any natural number including 1.  Let $q_{a_i}$, $p_{a_i}$, $c_{a_i}$ be the corresponding degrees and coefficients from their parameterizations
(\ref{curve}) in the local chart $(z,w)=(\frac1x,\frac yx)$. Let
$$c_i=\frac{c_{a_i}}{c_b}, \ W_i=Q_{p_{a_i},q_{a_i},c_i}(\theta)$$
be the corresponding constants and
polynomials from (\ref{rpq*}). Let $\theta_{ij}$, $i=1,\dots,l$, $j=1,\dots,p_{a_i}$ denote the roots of the
polynomials $W_i$.

Let $\Gamma_b\subset\Gamma$ denote the union of those germs in $\Gamma$ that are centered at $A$, tangent
to $b$ and have Puiseux exponents greater than $r=r_b$ (some of them may be linear).

Let $k_1$ denote the number of those points $\xi(t)$ of the intersection $T_tb\cap\Gamma$,
for which $x(\xi(t))=o(x(t))$, as $t\to A$. These are exactly those points that either lie in $\Gamma_b$ and
have $x$-coordinates $o(x(t))$,  see (\ref{a>b2*}), or lie in those germs
in $\Gamma$ that are not centered at $A$. Let $k_2$ denote the number of those points of the
intersection  $T_tb\cap\Gamma_b$ whose $x$-coordinates are asymptotic to $\frac r{r-1}x(t)$, see (\ref{a>b1*}).

\begin{proposition} \label{prnw} Let $r=r_b\neq2$. The collection of  powers $\theta_{ij}^{q_{a_i}}$ of the above roots contains
exactly $k_1$ powers equal to 2 and exactly $k_2$ powers equal to $\frac{r-2}{r-1}$. The collection $M$ of the other powers
$\theta_{ij}^{q_{a_i}}\neq 2,\frac{r-2}{r-1}$ (each of the latter powers being taken with the total multiplicity of the
corresponding roots) is invariant under the symmetry of the line $\cc$ with respect to $1$.
\end{proposition}

\begin{proof} The  points of intersection $T_tb\cap\Gamma$ whose $x$-coordinates are  $o(x(t))$ should be symmetric
with respect to $t$ to other intersection points with  $x$-coordinates asymptotically
equivalent to $2x(t)$ and vice versa. The latter should be
points of intersection with the germs $a_i$. This follows from Proposition \ref{asynt} and the fact that they cannot be
points of the intersection $T_tb\cap\Gamma_b$. The latter statement follows from  (\ref{a>b1*}) and the inequality
$\frac r{r-1}\neq2$, which follows from the assumption that
$r\neq2$. This together with Corollary \ref{casym} and Proposition \ref{asynt} implies that  exactly
$k_1$ powers $\theta_{ij}^{q_{a_i}}$ are equal to 2. Similarly, if the collection $\Gamma_b$ is non-empty,
then the intersection $T_tb\cap\Gamma_b$ contains points whose $x$-coordinates are  asymptotic to $\frac r{r-1}x(t)$, by (\ref{a>b1*}). Vice versa, the points of the intersection $T_tb\cap\Gamma$
with the latter asymptotics lie in  $T_tb\cap\Gamma_b$.
This follows from Corollary \ref{casym}, Proposition \ref{asynt} and the fact that they cannot be points of intersection
with the germs $a_i$: no number $s=(\frac r{r-1})^{\frac1{q_{a_i}}}$ can be
a root of a polynomial $W_i$ with $c_i\neq0$. Indeed,
$$W_i(s)=(r-1)s^{p_{a_i}}-rs^{p_{a_i}}\left(\frac r{r-1}\right)^{-1}+c_i=c_i\neq0.$$
The above points of the intersection $T_tb\cap\Gamma_b$ should be symmetric to the points of the
intersection $T_tb\cap\Gamma$ with $x$-coordinates asymptotically
equivalent to $(2-\frac r{r-1})x(t)=\frac{r-2}{r-1}x(t)$. The latter points of intersection
lie in the  union of the germs $a_i$, by Proposition \ref{asynt} and
due to $\frac{r-2}{r-1}\neq \frac r{r-1}$.  Therefore,
exactly $k_2$ powers $\theta_{ij}^{q_{a_i}}$ are equal to $\frac{r-2}{r-1}$. The points of the intersection
$T_tb\cap\Gamma$ having $x$-coordinates asymptotically equivalent to $vx(t)$ with $v\neq0,2,\frac r{r-1},\frac{r-2}{r-1}$
are symmetric with respect to $t$, by relative symmetry property.  The  collection of the corresponding
asymptotic factors $v$, thus  symmetric with respect to 1, coincides with the collection $M$ of the powers
$\theta_{ij}^{q_{a_i}}\neq 2,\frac{r-2}{r-1}$, by the above arguments.
The proposition is proved.
\end{proof}

\begin{proposition} \label{pl2} Let $r>1$. Consider a collection
 \begin{equation}S_r=\{(p_i,q_i,c_i)\}_{i=1,\dots,N}, \ q_i,p_i\in\nn, \ p_i>q_i, \ \frac{p_i}{q_i}=r, \ c_i\in\cc\setminus\{0\},
 \label{sr}\end{equation}
 set $W_i(\theta)=(r-1)\theta^{p_i}-r\theta^{p_i-q_i}+c_i$. Let $\theta_{ij}$ ($j=1,\dots,p_i$)  denote the roots of the polynomials $W_i$.
Let $M$ denote the collection of those $q_i$-th powers of roots $\theta_{ij}$ that are different from 2 and $\frac{r-2}{r-1}$:
each power being taken with the total multiplicity of the corresponding roots. Let $M$ be invariant under the
symmetry of the line $\cc$ with respect to $1$.  Let  $k_1$ ($k_2$) denote the number of those
 pairs $(i,j)$, for which  $\theta_{ij}^{q_i}$ equals 2 (respectively, $\frac{r-2}{r-1}$). Set
 $$\Pi=\sum_ip_{i}=\text{ the cardinality of the collection of all the roots } \theta_{ij}.$$
 Then
\begin{equation}(r-2)\Pi=k_2-k_1(r-1).\label{rpk1}\end{equation}
\end{proposition}

\begin{proof}
The invariance of the collection $M$ under the symmetry with respect to $1$ implies that the sum of its elements equals the cardinality
$card(M)=\Pi-k_1-k_2$. On the other hand,
\begin{equation} \Pi-k_1-k_2=\sum_{x\in M}x=\sum_{i,j}\theta_{ij}^{q_{i}}-2k_1-\frac{r-2}{r-1}k_2,\label{k12}\end{equation}
by definition. Let us calculate the latter right-hand side. One has
\begin{equation}\sum_{ij}\theta_{ij}^{q_{i}}=\sum_i\frac{p_{i}}{r-1}=\frac{\Pi}{r-1},\label{pt}\end{equation}
as in (\ref{sumtq}). Substituting (\ref{pt}) to (\ref{k12}) yields
$$\frac{\Pi}{r-1}-2k_1-\frac{r-2}{r-1}k_2=\Pi-k_1-k_2,$$
which is equivalent to (\ref{rpk1}).
\end{proof}

  \begin{proposition} \label{propk12}  Let $r>1$. Consider a collection $S_r$ as in (\ref{sr}), set
$W_i(\theta)=(r-1)\theta^{p_i}-r\theta^{p_i-q_i}+c_i$. Let $\theta_{ij}$ ($j=1,\dots,p_i$)  denote the roots of the polynomials $W_i$,
set $\Pi=\sum_{i=1}^Np_i$. Let $k_1$ ($k_2$) denote the number of index pairs $(i,j)$ for which $\theta_{ij}^{q_{i}}$  equals 2
(respectively, $\frac{r-2}{r-1}$). Let equality (\ref{rpk1}) hold.  Then $r=2$.
\end{proposition}

\begin{proof}
Let us represent $r$ as an irreducible fraction $r=\frac pq$, $p,q\in\nn$. One has
\begin{equation}p_{i}=s_ip, \ q_{i}=s_iq, \ s_i=\gcd(p_{i},q_{i}). \label{pqi}\end{equation}
Suppose the contrary: equality (\ref{rpk1}) holds and $r\neq2$.

Case 1): $r>2$. Hence, $r-2\geq\frac1q$. One has
$$k_2\geq(r-2)\Pi\geq\frac1q\Pi>\frac1p\Pi,$$
by (\ref{rpk1}). This implies that there exists a polynomial $W_i$ for which more than $\frac1p$-th part of its roots have
$q_{i}$-th powers equal to $\frac{r-2}{r-1}$. Thus, the number of the latter roots  is no less than $s_i+1=\frac{p_{i}}p+1$.
We will show that the above $W_i$
cannot exist. Let it exist, and let us fix it. None of its roots $\theta_{ij}$ with $\theta_{ij}^{q_{i}}=\frac{r-2}{r-1}$ can be a multiple root.
Indeed, the derivative of the polynomial $W_i(\theta)$ equals
$$\theta^{p_{i}-q_{i}-1}(p_{i}(r-1)\theta^{q_{i}}-r(p_{i}-q_{i}))=p_{i}(r-1)
\theta^{p_{i}-q_{i}-1}(\theta^{q_{i}}-1),$$
since $r(p_{i}-q_{i})=rq_{i}(r-1)=p_{i}(r-1)$. Therefore,
the $q_{i}$-th powers of the roots of the derivative are equal to $0,1\neq\frac{r-2}{r-1}$. Hence, $W_i$ has  at least $s_i+1$
distinct roots $\theta=\theta_{ij}$ with $\theta_{ij}^{q_{i}}=\frac{r-2}{r-1}$. By definition, the latter roots satisfy the equality
$$\theta^{p_{i}-q_{i}}((r-1)\theta^{q_{i}}-r)+c_i=\theta^{p_{i}-q_{i}}((r-2-r)+c_i$$
\begin{equation}=-2\theta^{p_{i}-q_{i}}+c_i=0\label{theta0}\end{equation}
and have equal $q_{i}$-th powers.  Hence, their $p_{i}$-th powers are also equal, by (\ref{theta0}).
Therefore, the ratio of any two above roots is simultaneously a
$q_{i}$-th and $p_{i}$-th root of unity, and hence, an $s_i$-th root of unity, since $p_{i}=s_ip$, $q_{i}=s_iq$ and $p$, $q$
are coprime. Therefore, the number of roots under question is no greater than $s_i$.
We get a contradiction.

Case 2): $1<r<2$. One has
$$k_1\geq\Pi\frac{2-r}{r-1}=\Pi\frac{2q-p}{p-q}\geq\frac1{p-q}\Pi>\frac1p\Pi,$$
by (\ref{rpk1}). This implies that the there exists a polynomial $W_i$ that has at least $s_i+1$ roots whose
$q_{i}$-th powers are equal to 2, as in the previous case.
We then get a contradiction, as in the above discussion. The proposition is proved.
\end{proof}

\begin{lemma} \label{lemn} Let $r>1$, $S_r=\{(p_i,q_i,c_i)\}_{i=1,\dots,N}$ be a collection as in (\ref{sr}).
Set $W_i(\theta)=(r-1)\theta^{p_i}-r\theta^{p_i-q_i}+c_i$. Let $\theta_{ij}$ ($j=1,\dots,p_i$)  denote the roots of the polynomials $W_i$.
Let $M$ denote the collection of those $q_i$-th powers of roots $\theta_{ij}$ that are different from 2 and $\frac{r-2}{r-1}$:
each power being taken with the total multiplicity of the corresponding roots. Let $M$ be invariant under the
symmetry of the line $\cc$ with respect to $1$.
Then $r=2$.
\end{lemma}

Lemma \ref{lemn} follows from Propositions \ref{pl2} and \ref{propk12}. Lemma \ref{lemn} together with Proposition
\ref{prnw}  imply the statement of  Theorem \ref{tsym}.

\section{Invariants of singularities and Pl\"ucker formulas. Proofs of Theorems \ref{shust} and \ref{tabmult}}
Here we prove Theorem \ref{shust} in Subsection 4.2. Then we prove Theorem \ref{tabmult} in Subsection
4.3. Afterwards in Subsection 4.4 we state and prove a purely algebraic-geometric  Theorem \ref{corconic}
that is a direct consequence of Theorems \ref{tsym} and \ref{shust} that seems to be of
independent interest.

 The proof of Theorem \ref{shust} essentially uses general Pl\"ucker and genus formulas for plane algebraic curves; the corresponding background material is presented in Subsection 4.1. The main observation
is that the upper bound $2$ to the projective Puiseux exponents of all transverse
local branches of the curve
and Pl\"ucker formulas yield that the singularity invariants of the
considered curve must obey a relatively high lower bound.
On the other hand, the contribution of the points in the infinite line
$\oci$ appears to be not sufficient to fit that lower bound unless the curve is a conic.

\subsection{Invariants of plane curve singularities}
For the reader's convenience,
we recall here main definitions and formulas. Almost all this stuff is classically known
(see \cite[Chapter III]{BK}, \cite[\S10]{Mi}, and the modern exposition in \cite[Section I.3]{GLS}).

Let $(x_0:x_1:x_2)$ be homogeneous coordinates on $\cp^2$.
Let $\alpha\subset\cp^2$ be a reduced, irreducible curve of degree $d>1$, i.e., given by a homogeneous
square-free, irreducible polynomial
$F(x_0,x_1,x_2)$ of degree $d>1$. For any point $A\in\alpha$ and a sufficiently small closed
ball $V\subset\cp^2$ centered at $A$ the intersection $\alpha\cap V$
is topologically a bouquet of discs $b_i$, $1\le i\le r$, given by the local branches of the germ $(\alpha,A)$.

\smallskip
{\it (1) Multiplicity and dual multiplicity of a local branch.} Given a local branch $b_i$
of the curve $\alpha$ at $A\in\alpha$
and its Puiseux parametrization (\ref{curve}), the number $s(b_i)=q$ is called the {\it multiplicity}, and
the number $s^*(b_i)=p-q$ the
{\it dual multiplicity} of the branch $b_i$. Note that $s(b_i)$ is the intersection multiplicity of the
branch $b_i$ with a transversal line, while $s^*(b_i)+s(b_i)$ is the intersection multiplicity with the tangent line $T_Ab_i$.
Observe also that the subquadraticity condition for $b_i$
is equivalent to
the relation
\begin{equation}s^*(b_i)\le s(b_i)\ .\label{ed1}\end{equation}

\smallskip
{\it (2) $\delta$-invariant.} Let $f(x,y)=0$ be an equation of the germ $(\alpha,A)$ (just $F=0$ rewritten in local  coordinates $x,y$).
Then $\alpha_\eps:=\{f(x,y)=\eps\}\cap V$, for $0<|\eps|\ll1$, is a smooth surface with $r$ holes ({\it Milnor fiber}). The $\delta$-invariant
of the germ $(\alpha,A)$ admits several equivalent definitions and topologically can be defined as the genus of the closed
surface obtained by attaching a sphere with $r$ holes to the surface $\alpha_\eps$.
The genus formula in the form of Hironaka \cite{Hir} reads
$$\frac{(d-1)(d-2)}{2}=g(\alpha)+\sum_{A\in\Sing(\alpha)}\delta(\alpha,A)\ ,$$
where $g(\alpha)$ is the geometric genus of $\alpha$, i.e., the genus of the Riemann surface
obtained by the resolution of singularities of $\alpha$. In particular, we have
\begin{equation}\sum_{A\in\Sing(\alpha)}\delta(\alpha,A)\le\frac{(d-1)(d-2)}{2}\ .\label{ed2}\end{equation}

\smallskip
{\it (3) Class of the singular point ($\kappa$-invariant).} Given a germ $(\alpha,A)$ and local affine coordinates
$(x,y)$, suppose that the $y$-axis is not tangent to any of the local branches $b_i$, $1\le i\le r$. Denote
 by $\alpha'$ the polar curve of $\alpha$ defined by the equation $\frac{\partial f}{\partial y}=0$.
The class of the germ $(\alpha,A)$ is defined by
$$\kappa(\alpha,A)=(\alpha\cdot\alpha')_A\ ,$$ the intersection multiplicity of $\alpha$ and $\alpha'$ at $A$.
It is well-known (see, for instance, \cite[Propositions I.3.35 and I.3.38]{GLS}) that
\begin{equation}\kappa(\alpha,A)=2\delta(\alpha,A)+\sum_{i=1}^r(s(b_i)-1)\ .\label{ed3}\end{equation}

\smallskip
{\it (4) Hessian of the singular (or inflection)
point.} The {\it Hessian} $H_\alpha$ of the curve $\alpha$ is the curve given by the equation
$\det\left(\frac{\partial^2F}{\partial x_i\partial x_j}\right)_{0\le i,j\le2}=0$. The {\it Hessian} of the germ $(\alpha,A)$ is
$$h(\alpha,A)=(\alpha\cdot H_\alpha)_A\ ,$$ the intersection multiplicity of $\alpha$ and $H_\alpha$ at
$A$. It vanishes in all smooth points of $\alpha$, where $\alpha$ quadratically intersects its tangent line.
An expression for $h(\alpha,A)$ via the preceding invariants was found in \cite[Formula (2)]{shus}. It
can be written as
\begin{equation}h(\alpha,A)=3\kappa(\alpha,A)+\sum_{i=1}^r(s^*(b_i)-s(b_i))\ .\label{ed6}\end{equation}
In view of $\deg H_\alpha=3(d-2)$, B\'ezout's theorem yields (a Pl\"ucker formula)
\begin{equation}3d(d-2)=\sum_{A\in \alpha}h(\alpha,A)\ .\label{ed7}\end{equation}

\subsection{Proof of Theorem \ref{shust}}
Let $\alpha\subset\cp^2$ be a curve of degree $d\ge2$, satisfying the hypotheses of Theorem \ref{shust}.
We will show that $d=2$.

Observe that $\delta(\alpha,A)=\kappa(\alpha,A)=h(\alpha,A)=0$ for all points $A\in\alpha\setminus\oci$.
Denote by ${\mathcal B}_{tr}$, resp. ${\mathcal B}_{tan}$ the set of the local branches of $\alpha$ centered on $\oci$ and
transversal, resp. tangent to $\oci$. Relation (\ref{ed1}) holds for all the local branches $b\in{\mathcal B}_{tr}$,
since they are subquadratic, by the condition of Theorem \ref{shust}. Thus, their contributions to the sum
$\sum_i(s^*(b_i)-s(b_i))$ are non-positive. Therefore, from
(\ref{ed6}) and (\ref{ed7}), we get
\begin{equation}3d(d-2)\le 3\sum_{A\in\alpha\cap\oci}\kappa(\alpha,A)+\sum_{b\in{\mathcal B}_{tan}}(s^*(b)-s(b))\ .
\label{ed11}\end{equation}
Together with (\ref{ed2}) and (\ref{ed3}) this yields
$$3d(d-2)\le 6\sum_{A\in\alpha\cap\oci}\delta(\alpha,A)+3\sum_{b\in{\mathcal B}_{tr}\cup{\mathcal B}_{tan}}(s(b)-1)+
\sum_{b\in{\mathcal B}_{tan}}(s^*(b)-s(b))$$
$$\le3(d-1)(d-2)+3\sum_{b\in{\mathcal B}_{tr}\cup{\mathcal B}_{tan}}(s(b)-1)+
\sum_{b\in{\mathcal B}_{tan}}(s^*(b)-s(b))$$
\begin{eqnarray}&=&3(d-1)(d-2)+\sum_{b\in{\mathcal B}_{tr}\cup{\mathcal B}_{tan}}(s(b)-1)+\sum_{b\in
{\mathcal B}_{tr}}s(b)\nonumber\\
& &+\sum_{b\in{\mathcal B}_{tr}}(s(b)-2)+\sum_{b\in{\mathcal B}_{tan}}(s^*(b)+s(b)-2)\ .
\label{ed8}\end{eqnarray}
Developing $d=(\alpha\cdot\oci)$ into contributions of local branches $b\in{\mathcal B}_{tr}\cup{\mathcal B}_{tan}$,
we obtain
\begin{equation}\begin{cases}&\sum_{b\in{\mathcal B}_{tr}\cup{\mathcal B}_{tan}}(s(b)-1)= d-|{\mathcal B}_{tr}\cup{\mathcal B}_{tan}|-
\sum_{b\in{\mathcal B}_{tan}}s^*(b)\\
&\qquad\qquad\qquad\qquad\qquad\le d-|{\mathcal B}_{tr}|-2|{\mathcal B}_{tan}|,\\
&\sum_{b\in
{\mathcal B}_{tr}}s(b)=d-\sum_{b\in{\mathcal B}_{tan}}(s^*(b)+s(b))\le d-2|{\mathcal B}_{tan}|,\\
&\sum_{b\in{\mathcal B}_{tr}}(s(b)-2)+\sum_{b\in{\mathcal B}_{tan}}(s^*(b)+s(b)-2)\\
&\qquad\qquad\qquad=d-2|{\mathcal B}_{tr}\cup{\mathcal B}_{tan}|,
\end{cases}\label{ed10}\end{equation}
and hence the sequence of relations (\ref{ed8}) reduces to
\begin{equation}2\geq|{\mathcal B}_{tr}|+2|{\mathcal B}_{tan}|\ .\label{ed9}\end{equation}

If ${\mathcal B}_{tan}=\emptyset$ and all the branches $b\in{\mathcal B}_{tr}$ are centered at one point, then
$\sum_{b\in{\mathcal B}_{tr}}s(b)=d$, and the intersection multiplicity of $\alpha$ with the
tangent to one of the branches $b\in{\mathcal B}_{tr}$
appears to be greater than $d$. This implies that the latter tangent line is
contained in $\alpha$. Thus, $\alpha$ splits off a line,
contrary to the irreducibility assumption.

If ${\mathcal B}_{tan}=\emptyset$, $|{\mathcal B}_{tr}|=2$, and the two branches $b_1,b_2\in{\mathcal B}_{tr}$ have
distinct centers, we have an equality in (\ref{ed9}); hence, equalities in all the above relations, in particular,
in (\ref{ed11}). Thus, in view of (\ref{ed6}), (\ref{ed7})
and the inequality $s^*(b_i)\leq s(b_i)$ (subquadraticity), it means $s^*(b_i)=s(b_i)$, $i=1,2$. Intersecting
$\alpha$ with the tangent lines
to $b_1$ and $b_2$, we obtain $s(b_i)\le\frac{d}{2}$, $i=1,2$, while the intersection with $\oci$ yields $s(b_1)+s(b_2)=d$.
It follows that $s(b_1)=s(b_2)=\frac{d}{2}$. Choosing affine coordinates in $\cp^2\setminus\oci$ so that the
coordinate axes are tangent to $b_1$ and $b_2$ (at infinity) respectively, we obtain that the Newton polygon of the
defining polynomial
of $\alpha$ is just the segment $[(0,0),(d/2,d/2)]$.
Indeed, in local affine coordinates $z_1,w_1$ in a neighborhood of the center $A$ of the branch $b_1$ such that the tangent
line $T_Ab_1$ is the $z_1$-axis and $\oci$ is the $w_1$-axis, we have $b_1$ given by
$$z_1=t^{d/2},\quad w_1=c_1t^d(1+o(1)),\quad t\in(\cc,0)\,$$ which means that the Newton diagram of
$\alpha$ in these coordinates is the segment $[(d,0),(0,d/2)]$, i.e., the coefficients of all the monomials
$z_1^iw_1^j$ with $(i,j)$ below this segment vanish. In the coordinates $x=\frac{1}{z_1}$, $y=\frac{w_1}{z_1}$, this yields that
the coefficients of all monomials $x^iy^j$ with $j<i$ vanish. The same consideration with the affine coordinates
$z_2,w_2$ in a neighborhood of the center $B$ of the branch $b_2$ such that $T_Bb_2$ is the $z_2$ -axis and $\oci$ is the $w_2$-axis
leads to the conclusion that the coefficients of all monomials $x^iy^j$ with $j>i$ vanish. This
finally leaves the
only Newton segment $[(0,0),(d/2,d/2)]$.
Note that a polynomial with such a Newton segment factors into the product of
$\frac{d}{2}$ binomials of type $xy-\lambda$. Thus, $d=2$ due to the irreducibility of $\alpha$.

If $|{\mathcal B}_{tan}|=1$ and $|{\mathcal B}_{tr}|=0$, then we have an equality in (\ref{ed9}); hence, all
the above inequalities turn to be equalities, in particular, the second relation in (\ref{ed10}), that is, $s^*(b)+s(b)=2$ for the unique branch of
$\alpha$ centered on $\oci$, which finally means that $d=2$.

\subsection{Proof of Theorem \ref{tabmult}}

Let $\gamma$ denote the complex projective Zariski closure  of the curve $C$. Let $\alpha$ be  its
arbitrary
nonlinear irreducible component. The curve $\alpha$ has neither singular, nor inflection
points in $\cc^2$ (Theorem \ref{relsym}),
and each its local branch transverse to the infinity line (if any) is quadratic, by Corollary \ref{cortsym}.
Therefore, $\alpha$ is a conic, by Theorem \ref{shust}. Thus, the curve $\gamma$ is a finite union of conics and
lines, and $C$ is a union of arcs of conics and lines. The latter union of arcs is finite: each potential end of
an arc should be a singular point of the curve $\gamma$, and the number of singular points of an algebraic
curve is finite. At least one conical arc is present, since $C$ is a $C^4$-smoothly immersed curve that
does not lie in a line. The parameter interval (circle) of the curve $C$ is thus split into a finite number
of segments, each of them parameterizes an entire arc of conic (line) in $C$.
Any two  arcs parameterized by adjacent segments
have contact of order at least 5, since  $C$ is a $C^4$-smoothly immersed curve. Therefore,
no conical arc can be adjacent to a linear arc, which implies that there are no linear arcs at all.
No two arcs of distinct conics can be adjacent neither: otherwise, their intersection index would be
greater than 4, by the above statement. This implies that $C$ lies in just one conic and proves Theorem
\ref{tabmult}.

\subsection{Appendix: a general corollary of Theorems \ref{tsym} and \ref{shust}}
The following theorem is a direct consequence of Theorems \ref{tsym} and \ref{shust}.

\begin{theorem} \label{corconic} Let  $\alpha\subset\cp^2$  be an irreducible algebraic curve
distinct from a line such that all its singular and inflection points (if any)
lie in the infinity line $\oci$. Let for every point  $A\in\alpha\cap\oc_{\infty}$ each local branch $\beta$ of the curve
$\alpha$ at $A$  that is transverse to the infinity line $\oc_{\infty}$ (if any) have the local relative symmetry property with respect to some $(T_A\beta,A)$-local
multigerm $\Gamma=\Gamma(\beta)$. Then $\alpha$ is a conic.
\end{theorem}

\begin{proof} Each local branch $\beta$ as above is quadratic (Theorem \ref{tsym}). Hence, $\alpha$ is a conic, by Theorem \ref{shust}.
 \end{proof}

\section{Acknowledgements}

We are grateful to Misha Bialy, Andrey Mironov and Sergei Tabachnikov for introducing us to polynomially integrable billiards, helpful discussions and providing
the starting point for our work: Theorem \ref{relsym}. This work was partly done during  the visits of the
first author (A.Glutsyuk) to   Sobolev Institute at Novosibirsk and to Tel Aviv University. He wishes to thank Andrey Mironov and
Misha Bialy for their invitations and hospitality and both institutions for their hospitality and support.
We are grateful to the referee for a very careful reading of
the paper and very helpful remarks, which also led to major improvement of the first author's paper \cite{gl}.

\end{document}